\documentclass{elsarticle}
\usepackage{amsmath}
\usepackage{amssymb}
\usepackage{amsthm}
\usepackage{amsfonts}
\usepackage{mathtools}
\usepackage{geometry}
\usepackage{hyperref}
\hypersetup{
    colorlinks=true,
    linkcolor=blue,
    filecolor=magenta,      
    urlcolor=cyan,
    pdftitle={Overleaf Example},
    pdfpagemode=FullScreen,
    }
\usepackage{pb-diagram}
\usepackage{comment}

\theoremstyle{plain}
\newtheorem{theorem}{Theorem}[section]
\newtheorem{corollary}[theorem]{Corollary}

\newtheorem{proposition}[theorem]{Proposition}
\newtheorem{lemma}[theorem]{Lemma}
\theoremstyle{definition}
\newtheorem{definition}[theorem]{Definition}
\newtheorem{example}[theorem]{Example}
\newtheorem{notation}[theorem]{Notation}
\newtheorem{remark}[theorem]{Remark}

\newtheorem{question}[theorem]{Question}

\newcommand{{\vp}}{{\varphi}}
\newcommand{{\x}}{{\mathbf{x}}}
\newcommand{{\ii}}{{\mathbf{i}}}
\newcommand{{\jj}}{{\mathbf{j}}}
\newcommand{{\ki}}{{\mathbf{k}}}
\newcommand{{\kk}}{{\mathbb{K}}}
\newcommand{{\hh}}{{\mathbb{H}}}
\newcommand{{\rr}}{{\mathbb{R}}}
\newcommand{{\rx}}{{\mathbb{R}[\mathbf{x}]}}
\newcommand{{\kx}}{{\mathbb{K}[\mathbf{x}]}}
\newcommand{{\hx}}{{\mathbb{H}[\mathbf{x}]}}
\newcommand{{\auton}}{{automorphically normalizable}}

\begin{document}

\begin{frontmatter}

\title{Noether's normalization in skew polynomial rings}

\author[OPENU]{Elad Paran}
\ead{paran@openu.ac.il}

\author[TDTU]{Thieu N. Vo\corref{mycorrespondingauthor}}
\cortext[mycorrespondingauthor]{Corresponding author}
\ead{vongocthieu@tdtu.edu.vn}


\address[OPENU]{The Open University of Israel}

\address[TDTU]{Fractional Calculus, Optimization and Algebra Research Group, Faculty of Mathematics and Statistics, Ton Duc Thang University, Ho Chi Minh City,
Vietnam}


\begin{abstract}
We study Noether's normalization lemma for finitely generated algebras over a division algebra.
In its classical form, the lemma states that if $I$ is a proper ideal of the ring $R=F[t_1,\ldots,t_n]$ of polynomials over a field $F$, then the quotient ring $R/I$ is a finite extension of a polynomial ring over $F$.
We prove that the lemma holds when $R=D[t_1,\ldots,t_n]$ is the ring of polynomials in $n$ central variables over a division algebra $D$. We provide examples demonstrating that Noether's normalization may fail for the skew polynomial ring $D[t_1,\ldots,t_n;\sigma_1,\ldots,\sigma_n]$ with respect to commuting automorphisms $\sigma_1,\ldots,\sigma_n$ of $D$.
We give a sufficient condition for $\sigma_1,\ldots,\sigma_n$ under which the normalization lemma holds for such ring.  In the case where $D=F$ is a field, this sufficient condition is proved to be necessary.

\end{abstract}

\end{frontmatter}

\section{Introduction}\label{sec:intro}

Let $D$ be a division algebra, and suppose that $R$ is a quotient ring of the polynomial ring $D[x_1,\ldots,x_n]$ in $n$ central variables over $D$ by a proper two-sided ideal $I$. We shall call such a ring $R$ a {\it centrally finitely generated} extension of $D$. The motivating question for this paper is the following: What can be said of the structure of such rings? In the case where $D$ is a field, this is a classical question, and the answer is given by Noether's normalization lemma \cite[Theorem 13.3]{Eisenbud2013}, which states that $R$ itself must be a finite extension of a polynomial ring over $D$.

If $D$ is an arbitrary division algebra and $R$ is a simple (left) $D[x_1,\ldots,x_n]$-module (equivalently, $R$ is a quotient module of $D[x_1,\ldots,x_n]$ by a maximal left ideal $I$), a theorem of Amitsur and Small \cite[Theorem 1]{AmitsurSmall78} states that $R$ is finite over a polynomial ring in central variables over $D$. Our first observation is closely related, stating the following generalization of Noether's lemma:

\begin{theorem}\label{main1} Let {\color{black} $D$ be a division algebra and} $R = D[x_1,\ldots,x_n]/I$ be a centrally finitely generated extension of $D$. Then $R$ is finite as a left module over a subring which is isomorphic to a polynomial ring in central variables over $D$. \end{theorem}

Let us emphasize that in the mentioned theorem of Amitsur-Small, the ideal $I$ is one-sided, but maximal, while in Theorem \ref{main1} the ideal $I$ is an arbitrary two-sided ideal.



Next, we consider a more general situation: Given commuting\footnote{Throughout this work, whenever we refer to a sequence of objects as ``commuting," we mean ``pairwise commuting".} automorphisms $\sigma_1,\ldots,\sigma_n$ of a division algebra $D$, we denote by $D[t_1,\ldots,t_n ; \sigma_1,\ldots,\sigma_n]$ the corresponding multivariate skew polynomial ring (see \S\ref{sec:prelim} below for a brief survey of basic facts concerning such rings). We say that a ring extension $S$ of $D$ is {\it automorphically finitely generated} if $S$ is a quotient ring of $D[t_1,\ldots,t_n ; \sigma_1,\ldots,\sigma_n]$ by a proper two-sided ideal for some commuting automorphisms $\sigma_1,\ldots,\sigma_n$ of $D$. The main goal of this paper is to study the structure of such rings. To that end, we introduce the following terminology:

\begin{definition} Let $D$ be a division algebra.
\begin{enumerate}
    
\item We say that a ring extension $S$ of $D$ is \textbf{automorphically normalizable} over $D$ if $S$ is finite as a left module over a subring which is isomorphic to $D[t_1,\ldots,t_n;\sigma_1,\ldots,\sigma_n]$, for some commuting automorphisms $\sigma_1,\ldots,\sigma_n$ of $D$. If these automorphisms can all be chosen as the identity, we say that $S$ is \textbf{centrally normalizable} over $D$.

\item A tuple $(\sigma_1,\ldots,\sigma_n)$ of commuting automorphisms in $\mathbf{Aut}(D)$ is called \textbf{automorphically} (resp. \textbf{centrally}) \textbf{normalizable} over $D$ if every quotient ring of $D[t_1,\ldots,t_n;\sigma_1,\ldots,\sigma_n]$ is automorphically (resp. centrally) normalizable over $D$.
\end{enumerate}
\end{definition}

Using this terminology, Theorem \ref{main1} above states that every centrally finitely generated extension over a division algebra $D$ is centrally normalizable over $D$. In the case where $D$ is centrally finite (that is, of finite dimension over its center), we deduce from Proposition~\ref{main} that Noether's normalization holds for arbitrary subrings of polynomial rings that are finitely generated over $D$:

\begin{theorem}[{{\color{black} See Proposition~\ref{main} below}}]\label{main_centrally_finite} Let $D$ be a centrally finite division algebra and
let $R=D[t_1,\ldots,t_n]$ be the ring of polynomials in $n$ central variables over $D$.
Then for every finite subset $A$ of $R$, the ring $S=D[A]$ generated by $A$ over $D$ is centrally normalizable over $D$. 
\end{theorem}

The key point in the proof of this result is that such a finitely generated subring $S$ must be {\bf centrally} finitely generated, which a priori is not obvious -- we obtain this a consequence of a theorem of Wilczy{\' n}ski \cite{Wil2014}.

Noether's normalization lemma, in its classical form, can be stated as follows: \textit{The tuple $(\mathbf{id}_F,\ldots,\mathbf{id}_F)$ of finitely many identity maps of a field $F$ is centrally normalizable over $F$}. We aim to identify those tuples of automorphisms of a division algebra $D$ that are automorphically normalizable over $D$.

Given a division algebra $D$, we denote its center by $\mathcal{Z}(D)$, and given an automorphism $\sigma$ of $D$, we denote by $D_{\sigma}$ the fixed ring of $D$. We prove:

\begin{theorem}[{{\color{black} See Theorem~\ref{thm:normalization_one_auto} below}}]\label{constant_tuple}
Let $D$ be a division algebra and let $\sigma$ be an automorphism of $D$ such that the field $F=\mathcal{Z}(D) \cap D_{\sigma}$ is infinite. Then the constant tuple $(\sigma,\ldots,\sigma)$, of any finite length, is automorphically normalizable over $D$.\end{theorem}

As consequence of Theorem \ref{constant_tuple}, we have the following corollary:

\begin{corollary}[{{\color{black} See Corollary~\ref{cor:normalization_cyclic_positive} below}}]\label{finite_orders} Let $D$ be a division algebra and let $\sigma_1,\ldots,\sigma_n$ be commuting automorphisms of $D$ such that $F=\mathcal{Z}(D) \cap \big( \cap_{i=1}^n D_{\sigma_i} \big)$ is infinite. If $\sigma_1^{d_1}=\ldots=\sigma_n^{d_n}$ for some positive integers $d_1,\ldots,d_n$, then the tuple $(\sigma_1,\ldots,\sigma_n)$ is automorphically normalizable over $D$. In particular, if $\sigma_1,\ldots,\sigma_n$ are of finite orders, then the tuple $(\sigma_1,\ldots,\sigma_n)$ is centrally normalizable over $D$.
\end{corollary}

This result gives a sufficient condition for a tuple of automorphisms of a division algebra $D$ to be automorphically normalizable over $D$. The authors do not know whether this condition is generally also a necessary one, see Question \ref{q3} below. However, in the case where $D$ is a field, we show that it is, provided that the condition that $\mathcal{Z}(D) \cap \big( \cap_{i=1}^n D_{\sigma_i} \big)$ is infinite holds. That is, we prove:

\begin{theorem}[{{\color{black} See Theorem~\ref{thm:notrivial_case} below}}]\label{fields_iff}
Let $F$ be a field and let $\sigma_1,\ldots,\sigma_n$ be commuting automorphisms of $F$ such that $\cap_{i=1}^n F_{\sigma_i}$ is infinite.
Then the tuple $(\sigma_1,\ldots,\sigma_n)$ is automorphically normalizable over $F$ if and only if $\sigma_1^{d_1}=\ldots=\sigma_n^{d_n}$ for some positive integers $d_1,\ldots,d_n$. \end{theorem}

For general division algebras, we give the following example where automorphic normalizability fails: If $\sigma \in \mathbf{Aut}(D)$ is of infinite inner order, then the tuple $(\sigma,\sigma^{-1})$ is not automorphically normalizable over $D$ (Proposition \ref{prop:counter-example-1} below). Another negative example is given by Proposition \ref{prop:counter-example2} below. These examples lead to some natural open questions, see Question \ref{q1} and Question \ref{q2}.

Finally, we note that the notion of automorphic normalizability can be used to formulate the following generalization of Zariski's lemma:

\begin{proposition}[{{\color{black} See Proposition~\ref{Zari} below}}]\label{ZariMain} Let $S$ be a ring extension of a division algebra $D$ such that $S$ is automorphically normalizable over $D$. 
If $S$ is a division algebra, then $S$ is finite as a left module over $D$. 
\end{proposition} 
The results discussed above concerning automorphic normalizability yield various situations to which this version of Zariski's lemma applies.

The paper is organized as follows: In \S\ref{sec:prelim} we provide basic facts concerning skew polynomials rings, automorphically finitely generated extensions, normalizability, and the proof of Proposition \ref{ZariMain}. In \S\ref{sec_noether} we study normalizability of constant tuples of automorphisms, and prove Theorem \ref{constant_tuple} and Corollary \ref{finite_orders}. In \S\ref{sec_center} we study central normalizability and prove Theorem \ref{main1} and Theorem \ref{main_centrally_finite}. Finally, in \S\ref{sec:general_tuples} we study general tuples of automorphisms, introduce the above mentioned negative examples where normalizability does not hold, and prove Theorem \ref{fields_iff}.

\section{Preliminaries}\label{sec:prelim}

Throughout this paper all rings are assumed to be associative with unity. 
By a {\bf division algebra} or a {\bf division ring}, we mean a ring in which every nonzero element has a multiplicative inverse.  If $D$ is a division algebra of finite dimension over its center, one says that $D$ is a {\bf centrally finite} division algebra \cite[Definition~14.1]{Lam1991}.

\subsection{Automorphically finitely generated ring extensions}\label{subsect:a.f.g.}

Let $R$ be a (possibly non-commutative) ring and let $\sigma$ be an automorphism of $R$. The skew polynomial ring $R[t,\sigma]$ consists of all polynomials in the variable $t$, with the usual addition, and with multiplication determined by the rule $ta = a^\sigma t$ for all $a \in R$ \cite[Chapter 2]{Goodearl2004}, \cite[Chapter 1]{Ore33}. The study of skew polynomial rings is prolific, both from a theoretical point of view, and for their various applications, see for example 
\cite{AP20b, BU09, Goo92, GL94b, LaL04}.

More generally, if $\sigma_1,\ldots,\sigma_n$ are automorphisms of $R$, one forms the multivariate skew polynomial ring $R[t_1,\ldots,t_n ; \sigma_1,\ldots,\sigma_n]$, in which the variables $t_1,\ldots,t_n$ commute pairwise, and multiplication is determined by the rule $x_i a = a^{\sigma_i}x_i$ for all $1 \leq i \leq n$ and $a \in R$. If one further assumes that the $\sigma_i$ commute pairwise, then one may view $R[t_1,\ldots,t_n ; \sigma_1,\ldots,\sigma_n]$ as an iterated skew polynomial ring: 
$$R[t_1,\ldots,t_n ; \sigma_1,\ldots,\sigma_n] = R[t_1,\ldots,t_{n-1} ; \sigma_1,\ldots,\sigma_{n-1}][t_n,\sigma_n],$$
by extending $\sigma_n$ trivially to an automorphism of $R[t_1,\ldots,t_{n-1} ; \sigma_1,\ldots,\sigma_{n-1}]$, given by $\sigma_n(t_i) = t_i$ for all $1 \leq i \leq n-1$ (see \cite[Theorem 2.4]{Vos86}). If the ring $R$ is (left) Noetherian then so is $S=R[t_1,\ldots,t_n;\sigma_1,\ldots,\sigma_n]$. Indeed, one can view $S$ as $R[t_1,\ldots,t_{n-1};\sigma_1,\ldots,\sigma_{n-1}][t_n;\sigma_{n}]$, by extending $\sigma_n$ to an automorphism of $R[t_1,\ldots,t_{n-1};\sigma_1,\ldots,\sigma_{n-1}]$ given by $\sigma_n(t_i) = t_i$ for each $1 \leq i < n$. Then, by applying inductively the skew Hilbert basis theorem given in \cite[page 17]{Goodearl2004}, we get that $S$ is also (left) Noetherian. As a consequence, if a ring extension of $S$ is finite over $S$ (i.e. finitely generated as a left $S$-module) then it is left integral over $S$ (see \cite[Lemma~2.2]{ParanVo2023}).

In the case where $\sigma_1,\ldots,\sigma_n$ are the identity automorphism, the ring $R[t_1,\ldots,t_n ; \sigma_1,\ldots,\sigma_n]$ is the usual polynomial ring $R[t_1,\ldots,t_n]$ in $n$ central variables over $R$.



\begin{definition}\label{def:f.g.}
Let $S/R$ be a ring extension. We say that $S$ is {\bf automorphically finitely generated} over $R$ if $S$ is isomorphic to a quotient ring $R[t_1,\ldots,t_n ; \sigma_1,\ldots,\sigma_n]/I$, where:
\begin{enumerate}
    \item $\sigma_1,\ldots,\sigma_n$ are commuting automorphisms of $R$,
    \item $I$ is a two-sided ideal of $R[t_1,\ldots,t_n ; \sigma_1,\ldots,\sigma_n]$, satisfying $R \cap I = \{0\}$. 
\end{enumerate}

We say that $S$ is {\bf centrally finitely generated} over $R$ if $S$ is automorphically finitely generated, and the automorphisms $\sigma_1,\ldots,\sigma_n$ can be chosen to be the identity automorphism.
\end{definition}

Note that in the case where $R$ is a division algebra, the condition that $R \cap I = \{0\}$ is equivalent to $I$ being a proper two-sided ideal. 
Thus in the case where $R$ is a field, a centrally finitely generated extension of $R$ is simply a finitely generated extension of $R$, in the classical sense. 

\begin{definition}[see {\cite[10.1.3,~p.~366]{McConnell2001}}]
Let $S/R$ be a ring extension. An element $a \in S$ is \textbf{automorphic} 
over $R$ if there exists an automorphism $\sigma \in \mathbf{Aut}(R)$ such that $ab=b^\sigma a$ for all $b \in R$. In this case, we say that $a$ is automorphic over $R$ with respect to $\sigma$. If, furthermore,  $\sigma=\mathbf{id}_R$, we say that $a$ \textbf{centralizes} $R$. 
\end{definition}

\begin{remark}
In general, if $S/R$ is a ring extension, an element $a \in S$ can be automorphic over $R$ with respect to several different automorphisms in $\mathbf{Aut}(R)$. However, if $\mathbf{Ann}_R(a):=\{r \in R \,|\, ra=0\}$ is equal to $\{0\}$, then the automorphism $\sigma$ is uniquely determined by $a$. Indeed, if $a$ is automorphic over $R$ with respect to two automorphisms $\sigma,\tau$, then for all $b \in R$ we have $b^\sigma a = ab= b^\tau a$, hence $(b^\sigma - b^\tau)a = 0$, hence $b^\sigma - b^\tau \in \mathbf{Ann}_R(a)$, hence $b^\sigma = b^\tau$.   
\end{remark}

In the case where $S=R=D$ is a division algebra, every element of $D$ is automorphic over $D$:

\begin{example}\label{ex:inner_automorphic}
	Let $D$ be a division algebra and $a$ an arbitrary nonzero element of $D$.
	We denote by $\mathbf{in}_a$ the \textbf{inner automorphism} of $D$ with respect to $a$ which is defined as 
	$\mathbf{in}_a(r)=ara^{-1}$	for all $r \in D$. Then $a$ is automorphic over $D$ with respect to $\mathbf{in}_a$, since $ar=ara^{-1} \cdot a=\mathbf{in}_a(r)a$. 
\end{example}

{\color{black}
The following lemma can be found in \cite{Goodearl2004}.
}

\begin{lemma}[{See also \cite[Proposition~2.4,~p.~37]{Goodearl2004}}]\label{isom} 
Let $S/R$ be a ring extension. Let $a_1,\ldots,a_n$ be $n$ commuting elements of $S$, automorphic over $R$ with respect to commuting automorphisms $\sigma_1,\ldots,\sigma_n \in \mathbf{Aut}(R)$. Then the natural embedding $R \xhookrightarrow{} S$ can be  uniquely extended to a ring homomorphism $\phi \colon R[t_1,\ldots,t_n; \sigma_1,\ldots,\sigma_n] \to S$ such that $\phi(t_i)=a_i$ for all $1 \leq i \leq n$.
\end{lemma}


Let $S/R$ be a ring extension and $A$ a subset of $S$.  The smallest subring of $S$ containing $R \cup A$, which will be denoted by $R[A]$, is the ring generated by $R\cup A$ inside $S$. In case $A=\{a_1,\ldots,a_n\}$ is a finite set, we shall usually write $R[a_1,\ldots,a_n]$ instead of $R[\{a_1,\ldots,a_n\}]$.

\begin{proposition}\label{equivalent} A ring extension $S/R$ is automorphically (centrally) finitely generated if and only if there exist commuting elements $a_1,\ldots,a_n$ in $S$, automorphic over $R$ (which centralize $R$) with respect to commuting automorphisms $\sigma_1,\ldots,\sigma_n$, such that $S = R[a_1,\ldots,a_n]$.
\end{proposition}
\begin{proof}
Suppose first that $a_1,\ldots,a_n \in S$ are automorphic commuting elements, with respect to commuting automorphisms $\sigma_1,\ldots,\sigma_n$, such that $S = R[a_1,\ldots,a_n]$.  Apply Lemma \ref{isom} to construct a homomorphism $\phi \colon R[t_1,\ldots,t_n; \sigma_1,\ldots,\sigma_n] \to S$ such that $\phi(t_i)=a_i$ for all $i$ and $\phi(r) = r$ for all $r \in R$. Consider the set of all elements of $S$ of the form:

$$\sum_{i_1,\ldots,i_n \geq 0} r_{i_1,\ldots,i_n}a_1^{i_1}\cdot \ldots \cdot a_n^{i_n},$$
with finitely many non-zero $r_{i_1,\ldots,i_n} \in R$. Since the $a_i$ commute, this set is clearly a subring of $S$ containing $R$ and $a_1,\ldots,a_n$, hence it equals $S = R[a_1,\ldots,a_n]$. Thus every element of $S$ is of the presented form, hence it is the image via $\phi$ of a corresponding element 
$$\sum_{i_1,\ldots,i_n \geq 0} r_{i_1,\ldots,i_n}t_1^{i_1}\cdot \ldots \cdot t_n^{i_n}$$
of $R[t_1,\ldots,t_n; \sigma_1,\ldots,\sigma_n]$. Thus $\phi$ is an epimorphism onto $S$, hence $$S = R[t_1,\ldots,t_n; \sigma_1,\ldots,\sigma_n]/\ker(\phi).$$
Since $\phi$ fixes $R$ point-wise, we have $R \cap \ker(\phi) = \{0\}$. Thus $S$ is an automorphically finitely generated extension of $R$. Moreover, if $a_1,\ldots,a_n$ centralize $R$, then we may assume that $\sigma_1,\ldots,\sigma_n$ are the identity automorphism, hence $S = R[t_1,\ldots,t_n]/\ker(\phi)$, which means that $S$ is a centrally finitely generated extension of $R$.

Conversely, if $S$ is an automorphically finitely generated extension of $R$, then by definition there exist commuting automorphisms $\sigma_1,\ldots,\sigma_n$ of $R$, such that $S = R[t_1,\ldots,t_n; \sigma_1,\ldots,\sigma_n]/I$, for some ideal $I$ satisfying $R \cap I = \{0\}$ . Then we may view $R$ as a subring of $S$, via the canonical map $r \mapsto r+I$. For each $1 \leq i \leq n$, let $a_i = t_i+I$. Then $a_1,\ldots,a_n$ generates $S$ over $R$ and are commuting, since $t_1,\ldots,t_n$ are commuting. We have 

$$a_ir = (t_i+I)(r+I) = t_ir+I = r^{\sigma_i}t_i+I = r^{\sigma_i}a_i$$
for all $1 \leq i \leq n$ and $r \in R$. Thus $a_1,\ldots,a_n$ are automorphic elements over $R$, with respect to automorphisms $\sigma_1,\ldots,\sigma_n$. Moreover, if $\sigma_1,\ldots,\sigma_n$ are the identity automorphism, then $a_1,\ldots,a_n$ centralize $R$, by the presented equality. \end{proof}

\begin{notation}\label{not:extensions}
Assume $S/R$ is an automorphically finitely generated ring extension and let $a_1,\ldots,a_n \in S$ be commuting elements, automorphic over $R$, such that $S=R[a_1,\ldots,a_n]$, as in Proposition \ref{equivalent}, with $\sigma_1,\ldots,\sigma_n$ being corresponding commuting automorphisms of $R$. Then we write $S=R[a_1,\ldots,a_n;\sigma_1,\ldots,\sigma_n]$. If in addition $\sigma_1=\ldots=\sigma_n$, then we write $S=R[a_1,\ldots,a_n;\sigma_1]$. If furthermore, $\sigma_1=\ldots=\sigma_n=\mathbf{id}_R$ (the case where $S$ is centrally finitely generated over $R$), then we write $S=R[a_1,\ldots,a_n]$.
\end{notation}

Note that, with the above notation, we have $S=R[a_1,\ldots,a_n;\sigma_1,\ldots,\sigma_n] = R[a_1,\ldots,a_n]$, that is, the ring $S$ is generated over $R$ by the set $\{a_1,\ldots,a_n\}$. In particular, if $S$ is the skew polynomial ring $R[t_1,\ldots,t_n;\sigma_1,\ldots,\sigma_n]$, then $S = R[t_1,\ldots,t_n]$ (indeed, as a ring, the skew polynomial ring $S$ is generated over $R$ by the variables $t_1,\ldots,t_n$). The notation $S=R[a_1,\ldots,a_n;\sigma_1,\ldots,\sigma_n]$ includes the extra information concerning the automorphisms that correspond to the automorphic elements $a_1,\ldots,a_n$.

\begin{remark}
Let $S/R$ be a ring extension. If two non zero-divisors $a_1,\,a_2 \in S$ are automorphic over $R$ with respect to automorphisms $\sigma_1,\,\sigma_2$, respectively, and if $a_1$ commutes with $a_2$, then $\sigma_1$ also commutes with $\sigma_2$.
Indeed, for every $b \in R$, we have $$b^{\sigma_1 \circ \sigma_2}a_1a_2 = (b^{\sigma_2})^{\sigma_1}a_1a_2 = a_1a_2b = a_2a_1b = (b^{\sigma_1})^{\sigma_2}a_2a_1 = b^{\sigma_2 \circ \sigma_1}a_2a_1 = b^{\sigma_2 \circ \sigma_1}a_1a_2,$$
which implies that $b^{\sigma_1 \circ \sigma_2} = b^{\sigma_2 \circ \sigma_1}$, since $a_1a_2$ is also a non-zero divisor.


\end{remark}

\subsection{Substitution maps and algebraic independence}

\begin{definition}\label{def:substituion} Let $S/R$ be a ring extension. Let $a_1,\ldots,a_n$ be $n$ commuting elements of $S$, automorphic over $R$ with respect to commuting automorphisms $\sigma_1,\ldots,\sigma_n \in \mathbf{Aut}(R)$. Let $$\phi \colon R[t_1,\ldots,t_n; \sigma_1,\ldots,\sigma_n] \to S$$ be the homomorphism given by Lemma~\ref{isom}. We call $\phi$ a \textbf{substitution map} at the point $a = (a_1,\ldots,a_n)$. For a polynomial $f \in R[t_1,\ldots,t_n;\sigma_1,\ldots,\sigma_n]$, we define the substitution of $f$ at $a=(a_1,\ldots,a_n)$, denoted by $f(a)$ or $f(a_1,\ldots,a_n)$, to be the image $\phi(f)$ in $S$.
\end{definition}


\begin{example}
Let $S = R = D$ be a division algebra and let $a_1,\ldots,a_n$ be commuting elements of $D^\times$.
Then each $a_i$ is automorphic over $D$ with respect to the inner automorphism $\sigma_i = \mathbf{in}_{a_i}$ (see Example~\ref{ex:inner_automorphic}), and these inner automorphisms are commuting.
Thus we have a substitution homomorphism $\phi \colon D[t_1,\ldots,t_n;\sigma_1,\ldots,\sigma_n] \to D$, given by $\phi(f) = f(a_1,\ldots,a_n) = \sum r_{\mathbf{i}} a_1^{i_1} \cdot \ldots \cdot a_n^{i_n}$ for all $f=\sum r_{\mathbf{i}} t_1^{i_1} \cdot \ldots \cdot t_n^{i_n} \in D[t_1,\ldots,t_n;\sigma_1,\ldots,\sigma_n]$. 
We note that in the case of one variable, i.e. $n=1$, the substitution here coincides with the $\sigma$-substitution $D[t_1,\sigma_1] \to D$ studied in \cite[\S2]{LL88}. Indeed, this follows immediately from the fact that $\sigma_1(a_1) = a_1$. We note, however, that the $\sigma$-substitution considered in \cite{LL88} is generally not a homomorphism, while in our situation it is.
\end{example}


\begin{example}\label{ex:substitution_poly}
    Let $\sigma_1,\ldots,\sigma_n$ be commuting automorphisms of a division algebra $D$ and let $R=S=D[t_1,\ldots,t_n;\sigma_1,\ldots,\sigma_n]$. We extend each $\sigma_i$ to an automorphism of $R$, determined by $\sigma_i(t_j)=t_j$ for all $i,j$. If $f_1,\ldots,f_n \in R$ are commuting and automorphic over $D$ with respect to automorphisms $\sigma_1,\ldots,\sigma_n$, then they induce the substitution map $f \mapsto f(f_1,\ldots,f_n)$ from $R$ to $R$. 
    
\end{example}

\begin{remark}
    In Definition~\ref{def:substituion}, the requirement that $a_1,\ldots,a_n \in S$ are commuting and automorphic over $R$ with respect to commuting automorphisms $\sigma_1,\ldots,\sigma_n$, is essential.
    Indeed, suppose that $a_1,\ldots,a_n$ are some elements of $S$.  
    Assume that there is a ring homomorphism $\phi$ from $R[t_1,\ldots,t_n;\sigma_1,\ldots,\sigma_n]$ to $S$ such that $\phi(r)=r$ for all $r \in R$ and $\phi(t_i)=a_i$ for all $i$. Then one verifies directly that the elements $a_1,\ldots,a_n \in S$ are commuting and automorphic over $R$ with respect to commuting automorphisms $\sigma_1,\ldots,\sigma_n$. 
\end{remark}

\begin{definition}\label{alg_ind}
Let $S/R$ be a ring extension and let $a_1,\ldots,a_n \in S$ be commuting elements. 
We say that $a_1,\ldots,a_n$ are {\bf (left) algebraically independent over $R$} if the set $$\left\{a_1^{i_1} \cdot \ldots \cdot a_n^{i_n} \right\}_{\mathbf{i}=(i_1,\ldots,i_n) \in (\mathbb{N}\cup\{0\})^n}$$
of monomials in $a_1,\ldots,a_n$ is (left) linearly independent over $R$. Otherwise, we say that $a_1,\ldots,a_n$ are {\bf (left) algebraically dependent over $R$}.
\end{definition}

If $S$ is a commutative ring, then the above definition coincides with the usual definition of algebraic independence. 


\begin{lemma} 
Let $S$ be a ring extension of a division algebra $D$.
Let $a_1,\ldots,a_n$ be $n$ commuting elements of $S$, automorphic over $D$, with respect to commuting automorphisms $\sigma_1,\ldots,\sigma_n \in \mathbf{Aut}(D)$. Then $a_1,\ldots,a_n$ are algebraically dependent over $D$ if and only if there exists a nonzero polynomial $f \in D[t_1,\ldots,t_n;\sigma_1,\ldots,\sigma_n]$ such that $f(a_1,\ldots,a_n)=0$.
\end{lemma} 
\begin{proof}
The claim follows from the definitions:  If $a_1,\ldots,a_n$ are algebraically dependent, then there are monomials  $M_1,\ldots,M_k$ in $a_1,\ldots,a_n$ and elements $b_1,\ldots,b_k \in D$, not all zero, such that $b_1M_1+\ldots+b_kM_k = 0$. Consider now the homomorphism $\phi \colon D[t_1,\ldots,t_n; \sigma_1,\ldots,\sigma_n] \to S$ given by Lemma~\ref{isom}. Since the $a_i$ commute, the monomials $M_1,\ldots,M_k$ are images via $\phi$ of respective monomials $N_1,\ldots,N_k$ in $t_1,\ldots,t_n$, and the $N_i$ are distinct since the $M_i$ are distinct. Thus $f = b_1N_1+\ldots+b_kN_k$ is a non-zero polynomial in $ D[t_1,\ldots,t_n;\sigma_1,\ldots,\sigma_n]$, and we have $\phi(f) = \sum b_i M_i = 0$. 

Conversely, if $f \in  D[t_1,\ldots,t_n;\sigma_1,\ldots,\sigma_n]$ is a non-zero polynomial then we may write it as $b_1N_1+\ldots+b_kN_k$, where the $N_i$ are monomials in $t_1,\ldots,t_k$ and the $b_i$ are not all zero. If $\phi(f) =0$ then  $\sum b_i \phi(N_i) = 0$, and so the $\phi(N_i)$ are left linearly dependent over $D$. But each $\phi(N_i)$ is a monomial expression in $a_1,\ldots,a_n$, by the definition of $\phi$, hence the $a_i$ are algebraically dependent. 
%
 \end{proof}



\begin{lemma}\label{lem:isom_to_skew}
    Let $S$ be a ring extension of a division algebra $D$.
    Let $a_1,\ldots,a_n$ be $n$ commuting elements of $S$, automorphic over $D$, with respect to commuting automorphisms $\sigma_1,\ldots,\sigma_n \in \mathbf{Aut}(D)$. If $a_1,\ldots,a_n$ are algebraically independent over $D$, then the subring $D[a_1,\ldots,a_n]$ of $S$ is isomorphic to the skew polynomial ring $D[t_1,\ldots,t_n; \sigma_1,\ldots,\sigma_n]$.
\end{lemma}

\begin{proof}
    It follows from Lemma~\ref{isom} and Definition~\ref{def:substituion} that the substitution map $\phi(f) = f(a_1,\ldots,a_n)$ from the skew polynomial ring $R=D[t_1,\ldots,t_n; \sigma_1,\ldots,\sigma_n]$ to $S$ is a ring homomorphism. Since $a_1,\ldots,a_n$ are algebraically independent over $D$, $\mathbf{ker}(f)=0$.
    Thus, $\phi$ is injective. Hence, $\phi$ induces an isomorphism from $R$ to its image $\mathbf{Im}(\phi)=D[a_1,\ldots,a_n]$, which is a subring of $S$.
\end{proof}

\subsection{Normalizability}

We now begin our study of the ``normalizability" property for automorphically finitely generated extensions of division algebras.

\begin{definition}\label{def:auto_normalizability}
Let $S$ be a ring extension of a division algebra $D$. We say that $S$ is \textbf{automorphically normalizable} over $D$ if there exists a non-negative integer $n$ and commuting elements $a_1,\ldots,a_n \in S$, automorphic over $D$ with respect to commuting automorphisms $\sigma_1,\ldots,\sigma_n \in \mathbf{Aut}(D)$, such that:
\begin{enumerate}
    \item $a_1,\ldots,a_n$ are left algebraically independent over $D$, and
    \item $S$ is finite as a left module over the subring $D[a_1,\ldots,a_n]$ of $S$.
\end{enumerate}
If in addition $\sigma_1=\ldots=\sigma_n=\mathbf{id}_D$, 
that is if $a_1,\ldots,a_n$ centralize $D$, then we say that $S$ is \textbf{centrally normalizable} over $D$.
\end{definition}

Note that, in the above definition, the subring $D[a_1,\ldots,a_n]$ of $S$ is isomorphic to the skew polynomial ring $D[a_1,\ldots,a_n;\sigma_1,\ldots,\sigma_n]$ (see Lemma~\ref{lem:isom_to_skew}). Thus, the automorphically normalizable extensions of $D$ are precisely those which are finite extensions of skew polynomial rings of the form $D[t_1,\ldots,t_n;\sigma_1,\ldots,\sigma_n]$.

\begin{definition}\label{def:sigma-normalizability}
Let $D$ be a division algebra. A tuple $(\sigma_1,\ldots,\sigma_n)$ of commuting automorphisms of $D$ is called \textbf{automorphically} (resp. \textbf{centrally}) \textbf{normalizable} over $D$ if every automorphically finitely generated extension of the form $S=D[a_1,\ldots,a_n;\sigma_1,\ldots,\sigma_n]$ of $D$ is automorphically (resp. centrally) normalizable over $D$. Equivalently, $(\sigma_1,\ldots,\sigma_n)$ is automorphically normalizable over $D$ if every quotient ring of $D[t_1,\ldots,t_n; \sigma_1,\ldots,\sigma_n]$ by a proper two-sided ideal is automorphically (resp. centrally) normalizable over $D$. 
\end{definition}

\begin{remark}\label{rem:normalizability}
In the above definition, the normalizability of a tuple is clearly independent of the order of the automorphisms. That is, if $\pi$ is a permutation in $S_n$, then the tuple $(\sigma_1,\ldots,\sigma_n)$ is automorphically (centrally) normalizable over $D$ if and only if the tuple $(\sigma_{\pi(1)},\ldots,\sigma_{\pi(n)})$ is automorphically (centrally) normalizable over $D$.
\end{remark}


We concludes this section with the following generalization of Zariski's lemma (Proposition \ref{ZariMain} of the introduction):

\begin{proposition}\label{Zari} 
Let $S$ be a ring extension of a division algebra $D$ such that $S$ is automorphically normalizable over $D$. 
If $S$ is a division algebra, then $S$ is finite over $D$. 
\end{proposition}

\begin{proof} 
Since $S$ is automorphically normalizable over $D$, 
there exist a non-negative integer $n$ and 
commuting elements $a_1,\ldots,a_n$ of $S$, automorphic over $D$ with respect to commuting automorphisms $\sigma_1,\ldots,\sigma_n \in \mathbf{Aut}(D)$, such that: the subring $R=D[a_1,\ldots,a_n]$ of $S$ is isomorphic to the skew polynomial ring $D[t_1,\ldots,t_n;\sigma_1,\ldots,\sigma_n]$, and $S$ is finite over $R$.

We claim that $n=0$.
Assume to the contrary, that $n>0$, then $a_1$ is not invertible in $R$ (but $a_1$ is invertible in $S$).
By a noncommutative form of Hilbert's Basis Theorem \cite[Theorem~1.17,~page~17]{Goodearl2004}, it follows that $R$ is left Noetherian.
Therefore, by \cite[Lemma~2.3]{ParanVo2023}, $S$ is integral over $R$.
In particular, the element $a_1^{-1} \in S$ is integral over $R$.
Thus, there exists a positive integer $k$ and elements $r_0,\ldots,r_{k-1} \in R$ such that
$$a_1^{-k}+r_{k-1}a_1^{-(k-1)}+\ldots+r_1a_1^{-1}+r_0=0.$$
By multiplying both sides by $a_1^{k-1}$, we obtain
$$a_1^{-1}=-\left(r_{k-1}+\ldots+r_1a_1^{k-2}+r_0a^{k-1}\right),$$
which is an element in $R$.
This contradicts to the fact that $a_1$ is not invertible in $R$.
Thus $n=0$ and $S$ is finite over $R=D$.
\end{proof}


\section{Normalizability of constant tuples}\label{sec_noether}


In the following, given a division algebra $D$, we denote by $D[t_1,\ldots,t_d;\sigma]$ the ring of skew polynomials, $D[t_1,\ldots,t_d;\sigma,\ldots,\sigma]$, for any $\sigma \in \mathbf{Aut}(D)$. 
In the case where $\sigma=\mathbf{id}_D$, we will drop $\sigma$ and write $D[t_1,\ldots,t_d]$ to denote the ring of polynomials in $d$ central variables over $D$ (see Notation~\ref{not:extensions}).
Let $D_{\sigma}$ be the fixed subring of $D$ via $\sigma$, i.e.
$D_{\sigma} = \{r \in D \,|\, \sigma(r)=r\}$.
We also denote $F=\mathcal{Z}(D) \cap D_{\sigma}$, where $\mathcal{Z}(D)$ the center of $D$.
The next result is a special case of \cite[Theorem 1.2]{Paran2023} and the combinatorial Nullstellensatz \cite{Michalek2010}. 


\begin{lemma}\label{comb} 
Let $D$ be a division algebra. Let $f \in D[t_1,\ldots,t_d]$ be a nonzero polynomial of total degree $m=\sum_{i=1}^d {k_i}$, where each $k_i$ is a non-negative integer, such that the coefficient of $t_1^{k_1}\cdot\ldots\cdot t_d^{k_d}$ in $f$ is non-zero. 
Let $A_1,\ldots,A_d$ be finite subsets of $F=\mathcal{Z}(D)$ such that $|A_i| > k_i$ for $i = 1,2,\ldots,d$. 
Then there is a point in $A_1 \times \dots \times A_d$ at which $f$ does not vanish. 
\end{lemma}

\begin{proof}
    See \cite[Theorem 1.2]{Paran2023}.
\end{proof}

\begin{lemma}\label{lem:affine_f} 
Let $D$ be a division algebra. Let $f \in D[t_1,\ldots,t_d]$ be a non-zero polynomial. 
If $F=\mathcal{Z}(D)$ is infinite, then there exist  $a_1,\ldots,a_d \in F$ such that $f(a_1,\ldots,a_d) \neq 0$. 
\end{lemma}

\begin{proof} 
Let $\lambda t_1^{k_1}\cdot\ldots\cdot t_d^{k_d}$ with be a monomial appearing in $f$ with  $\lambda \neq 0$ and $\deg(f) = k_1+\ldots+k_d$. Choose any subsets $A_1,\ldots,A_d$ of $F$ such that $|A_i| > k_i$ for each $i$, and apply Lemma~\ref{comb}.
\end{proof}

\begin{lemma}\label{lem:proj_f}
Let $D$ be a division algebra. Let $f \in D[t_1,\ldots,t_d]$ be a non-zero homogeneous polynomial. 
If $F=\mathcal{Z}(D)$ is infinite, then there exist $a_1,\ldots,a_{d-1} \in F$ such that $f(a_1,\ldots,a_{d-1},1) \neq 0$.
\end{lemma}\begin{proof}
Apply Lemma~\ref{lem:affine_f} for $f(t_1,\ldots,t_{d-1},1)$.
\end{proof}

\begin{lemma}\label{lem:ring_linear_subsituting}
    Let $D$ be a division algebra. Let $R=D[t_1,\ldots,t_d;\sigma]$ be a skew polynomial ring, and let $a_1,\ldots,a_{d-1} \in F=\mathcal{Z}(D) \cap D_{\sigma}$.
    Then the map
    $\phi:R \to R$ defined by $\phi(f(t_1,\ldots,t_d))=f(t_1+a_1t_d,\ldots,t_{d-1}+a_{d-1}t_d,t_d)$ is a ring homomorphism.
\end{lemma}

\begin{proof}
	For all $i=1,\ldots,n-1$ and $r \in D$, we have
	$$(t_i+a_it_d) r = r^\sigma t_i + a_i r^\sigma t_d = 
		r^\sigma t_i + r^\sigma a_i t_d
		= r^\sigma (t_i+a_it_d).$$
	Therefore, the elements $t_1+a_1t_d,\ldots,t_{d-1}+a_{d-1}t_d,t_d$ are all automorphic over $D$ with respect to $\sigma$.
    The claim then follows by Example~\ref{ex:substitution_poly}.
\end{proof}

\begin{lemma}\label{lem:poly_f}
Let $D$ be a division algebra. Let $f \in D[t_1,\ldots,t_d;\sigma]$ be a nonzero skew polynomial. 
If $F=\mathcal{Z}(D) \cap D_{\sigma}$ is infinite, then there exist $a \in D$ and $a_1,\ldots,a_{d-1}$ in $F$ such that the polynomial
$$g=a \cdot f(t_1+a_1t_d,\ldots,t_{d-1}+a_{d-1}t_d,t_d)$$
is a monic polynomial in $t_d$ with coefficients in $D[t_1,\ldots,t_{d-1};\sigma]$. (Note that we may view $D[t_1,\ldots,t_d;\sigma]$ as $D[t_1,\ldots,t_{d-1};\sigma][t_d;\sigma]$, as discussed in the beginning of \S\ref{subsect:a.f.g.}.)
\end{lemma}

\begin{proof}
Assume that $f=\sum\limits_{\mathbf{i} \in \mathbb{N}} b_{\mathbf{i}} t_1^{i_1} \cdot \ldots \cdot t_d^{i_d}$ is of degree $m$.
Let $\alpha_1,\ldots,\alpha_{d-1}$ be generic elements of $F$.
According to Lemma~\ref{lem:ring_linear_subsituting}, the substitution $f(t_1+\alpha_1t_d,\ldots,t_{d-1}+\alpha_{d-1}t_d,t_d)$ is well-defined.
For each $k=1,\ldots,d-1$, since $\alpha_k$ is in $F=\mathcal{Z}(D) \cap D_{\sigma}$, we have
\begin{align*}
  (t_k+\alpha_kt_d)^{i_k} &= (\alpha_kt_d)^{i_k} + {\text{ terms of lower degree in } t_d} \\
  &= \alpha_k \cdot \sigma(\alpha_k) \cdot \ldots \cdot \sigma^{i_k-1}(\alpha_k) t_d^{i_k} + {\text{ terms of lower degree in } t_d} \\
  &= \alpha_k^{i_k} t_d^{i_k} + {\text{ terms of lower degree in } t_d}.
\end{align*}
Therefore, 
\begin{align*}
  f(t_1+\alpha_1t_d,&\ldots,t_{d-1}+\alpha_{d-1}t_d,t_d) 
  =  \sum\limits_{\mathbf{i}} b_{\mathbf{i}} \left(\prod\limits_{k=1}^{d-1} (t_k+\alpha_kt_d)^{i_k}\right) \cdot t_d^{i_d}\\
  &= \sum\limits_{\mathbf{i}} b_{\mathbf{i}} \left(\prod\limits_{k=1}^{d-1} \left( \alpha_k^{i_k} t_d^{i_k} + {\text{ terms of lower degree in } t_d} \right) \right) \cdot t_d^{i_d}\\
  &= \sum\limits_{\mathbf{i}} b_{\mathbf{i}} \left( \prod\limits_{k=1}^{d-1} \alpha_k^{i_k} \right) \cdot t_d^{i_1+\ldots+i_d} + {\text{ terms of lower degree in } t_d}\\
  &=f_m(\alpha_1,\ldots,\alpha_{d-1},1) t_d^m + \text{ terms of lower degrees in } t_d.
\end{align*}
Here, $f_m=\sum_{\mathbf{i}} b_{\mathbf{i}} \left( \prod_{k=1}^{d-1} u_k^{i_k} \right)$ is a homogeneous polynomial in ring $D[u_1,\ldots,u_d]$ of central polynomials (thus the substitution $f_m(\alpha_1,\ldots,\alpha_{d-1},1)$ is well-defined).
Note that $f_m$ is identical to the homogeneous part of degree $m$ of $f$, except that the variables in $f_m$ centralize $D$. 
By Lemma~\ref{lem:proj_f}, there exist nonzero elements $a_1,\ldots,a_{d-1}$ in $F$ such that $f_m(a_1,\ldots,a_{d-1},1) \neq 0$.
Set $a=f_m(a_1,\ldots,a_{d-1},1)^{-1}$.
Then the polynomial $g=a \cdot f(t_1+a_1t_d,\ldots,t_{d-1}+a_{d-1}t_d,t_d)$ is monic in $t_d$.
\end{proof}

\begin{lemma}\label{lem:finite-over-normalizable}
Let $D$ be a division algebra and let $D \subseteq S \subseteq T$ be ring extensions.
If $S$ is automorphically (centrally) normalizable over $D$ and if $T$ is finite over $S$, then $T$ is also automorphically (centrally) normalizable over $D$.
\end{lemma}

\begin{proof}
Since $S$ is automorphically (centrally) normalizable over $D$, there exist commuting elements $a_1,\ldots,a_n$ in $S$ that are automorphic over $D$ (which centralize $D$), algebraically independent over $D$, and such that $S$ is finite over $R=D[a_1,\ldots,a_n]$. Since $T$ is finite over $S$ and $S$ is finite over $R$, $T$ is also finite over $R$.
Thus $T$ is automorphically (centrally) normalizable over $D$.
\end{proof}

We can now prove Theorem \ref{constant_tuple} of the introduction:

\begin{theorem}\label{thm:normalization_one_auto}
Let $D$ be a division algebra and let $\sigma$ be an automorphism of $D$ such that the field $F=\mathcal{Z}(D) \cap D_{\sigma}$ is infinite.
Then any constant tuple $(\sigma,\ldots,\sigma)$ is automorphically normalizable over $D$.
\end{theorem}

\begin{proof} 
Let $S=D[z_1,\ldots,z_n;\sigma]$ be an arbitrary automorphic extension of $D$. We must prove that $S$ is automorphically normalizable over $D$. We will do so by induction on $n$.
For $n = 0$ we have $S = D$ and there is nothing to prove. 

Assume we have proven the claim up to $n-1$, and let us prove it for $n$.
Since $z_1,\ldots,z_n$ are automorphic over $D$ with respect to $\sigma$ over $D$, the substitution map $f \mapsto f(z_1,\ldots,z_n)$ for skew polynomials in $D[t_1,\ldots,t_n;\sigma]$ is well-defined and is a ring homomorphism, by Lemma \ref{isom}. If the elements $z_1,\ldots,z_n$ are algebraically independent over $D$, then it follows from Lemma~\ref{lem:isom_to_skew} that $S$ is isomorphic to the skew polynomial ring $D[t_1,\ldots,t_n;\sigma]$. 
Thus $S$ is automorphically normalizable over $D$ (by choosing $d=n$ and $a_i=z_i$ in Definition~\ref{def:auto_normalizability}).

Let us consider the case when the elements $z_1,\ldots,z_n$ are algebraically dependent over $D$.
Then there is a nonzero skew polynomial $f\in D[t_1,\ldots,t_n;\sigma]$ such that $f(z_1,\ldots,z_n) = 0$.
By Lemma~\ref{lem:poly_f}, there exist an element $a \in D$, a tuple $(b_1,\ldots,b_{n-1}) \in F^{n-1}$, and a non-negative integer $m$ such that $g = af(t_1+b_1t_n,\ldots,t_{n-1}+b_{n-1}t_n,t_n)$ is a monic polynomial in $t_n$, of the form
\begin{equation}\label{eq:g}
g = t_n^m+\text{ terms in } (t_1,\ldots,t_n) \text{ of lower degree in }t_n.
\end{equation} 
Note that, since the elements $t_1+b_1t_n,\ldots,t_{n-1}+b_{n-1}t_n,t_n$ are commuting and automorphic over $D$ with respect to $\sigma$, the substitution of $f$ at $(t_1+b_1t_n,\ldots,t_{n-1}+b_{n-1}t_n,t_n)$ is well-defined (see Example~\ref{ex:substitution_poly}).

Set $\tilde{z}_1=z_1-b_1z_n,\ldots,\tilde{z}_{n-1}=z_{n-1}-b_{n-1}z_n$ and $\tilde{z}_n=z_n$.
Then the elements $\tilde{z}_i$ are commuting and automorphic over $D$ with respect to $\sigma$.
We can evaluate the skew polynomial $g$ at these elements and obtain
$$g(\tilde{z}_1,\ldots,\tilde{z}_n) = a f(\tilde{z}_1+b_1\tilde{z}_n,\ldots,\tilde{z}_{n-1}+b_{n-1}\tilde{z}_n,\tilde{z}_n) 
= a f(z_1,\ldots,z_n)=0.$$
Therefore, by substituting $(t_1,\ldots,t_n)$ at $(\tilde{z}_1,\ldots,\tilde{z}_n)$ in Eq.~\eqref{eq:g}, we get that
$$0 = \tilde{z}_n^m + \text{ terms in } (\tilde{z}_1,\ldots,\tilde{z}_n) \text{ of lower degree  in } \tilde{z}_n.$$
Thus $\tilde{z}_n$ is left integral over the sub-algebra $R = D[\tilde{z}_1,\ldots,\tilde{z}_{n-1}]$ of $S$. 
As a consequence, $S=D[z_1,\ldots,z_n]=D[\tilde{z}_1,\ldots,\tilde{z}_n]$ is finite over $R$.
By the induction hypothesis, $R$ is automorphically normalizable over $D$.
It follows from Lemma~\ref{lem:finite-over-normalizable} that $S$ is also automorphically normalizable over $D$.
\end{proof}



%
%

We can now prove Corollary \ref{finite_orders} of the introduction:

\begin{corollary}\label{cor:normalization_cyclic_positive}
Let $D$ be a division algebra and $\sigma_1,\ldots,\sigma_n$ be commuting elements in $\mathbf{Aut}(D)$ such that $F=\mathcal{Z}(D) \cap \big( \cap_{i=1}^n D_{\sigma_i} \big)$ is infinite.
If $\sigma_1^{d_1}=\ldots=\sigma_n^{d_n}$ for some positive integers $d_1,\ldots,d_n$, then the tuple $(\sigma_1,\ldots,\sigma_n)$ is automorphically normalizable over $D$.

In particular, if $\sigma_1,\ldots,\sigma_n$ are of finite orders, then the tuple $(\sigma_1,\ldots,\sigma_n)$ is centrally normalizable over $D$.
\end{corollary}

\begin{proof}
Let $S=D[z_1,\ldots,z_n;\sigma_1,\ldots,\sigma_n]$ be an arbitrary automorphic extension of $D$.
Set $\sigma=\sigma_1^{d_1}$, and for each $i=1,\ldots,n$, we set $w_i=z_i^{d_i}$.
Then for every $b \in D$, we have
\begin{align*}
w_ib = z_i^{d_i} b = \sigma_i^{d_i}(b) z_i^{d_i} = \sigma(b) w_i.  
\end{align*}
Therefore, $w_i$ is automorphic over $D$ with respect to $\sigma$.
Since, the elements $z_i$ commute, the elements $w_i$ also commute.
Thus, the ring $R=D[w_1,\ldots,w_n;\sigma]$ is automorphically finitely generated over $D$.
Note that $F \subseteq \mathcal{Z}(D) \cap D_{\sigma}$.
Therefore, $\mathcal{Z}(D) \cap D_{\sigma}$ is an infinite set.
It follows from Theorem~\ref{thm:normalization_one_auto} that $R$ is automorphic normalizable over $D$.
Observe that $S$ is finite over $R$, therefore, by Lemma~\ref{lem:finite-over-normalizable}, $S$ is automorphically normalizable over $D$.
Hence, the tuple $(\sigma_1,\ldots,\sigma_n)$ is automorphically normalizable over $D$.
\end{proof}

The above corollary gives a sufficient condition for a tuple of automorphisms in $\mathbf{Aut}(D)$ to be automorphically normalizable over $D$.
The question whether this sufficient condition is ``necessary" is still open.
In Theorem~\ref{thm:notrivial_case} below, we will prove that this condition is indeed ``necessary" in the case where $D=F$ is a field.

\section{Tuples of identity maps and central extensions}\label{sec_center}

In Theorem~\ref{thm:normalization_one_auto}, if the tuple of automorphisms is the identity tuple $(\mathbf{id}_D,\ldots,\mathbf{id}_D)$, (i.e. $\sigma=\mathbf{id}_D$), we will prove that the ``infinite" condition can be dropped.


\begin{lemma}\label{lem:ring_nonlinear_subsituting}
    Let $D$ be a division algebra, let $R=D[t_1,\ldots,t_n]$, and let $d_1,\ldots,d_{n-1}$ be non-negative integers.
    Then the map
    $\phi:R \to R$ defined by $\phi(f(t_1,\ldots,t_n))=f(t_1+t_n^{d_1},\ldots,t_{d-1}+t_n^{d_{n-1}},t_n)$ is a ring homomorphism.
\end{lemma}

\begin{proof}
    This follows directly from Example~\ref{ex:substitution_poly}.
\end{proof}

We will also need the following lemma, which is similar to \cite[(41.1),~page~44]{Nagata1975} and \cite[Lemma~13.2]{Eisenbud2013}.

\begin{lemma}\label{lem:F-finite}
Let $D$ be a division algebra and let $f \in D[t_1,\ldots,t_n]$ be a nonzero polynomial. 
Set $d=1+\deg f$ where $\deg f$ is the total degree of $f$.
Then there exists $a \in D$ such that the polynomial
$$g=a \cdot f\left(t_1+t_n^{d^{n-1}},t_2+t_n^{d^{n-2}},\ldots,t_{n-1}+t_n^{d},t_n\right)$$
is a monic polynomial in $t_n$ with coefficients in $D[t_1,\ldots,t_{n-1}]$.
\end{lemma}

\begin{proof}
We write $f = \sum_{I \in \mathcal{S}} a_{I} t_1^{i_1} t_2^{i_2} \ldots t_{n}^{i_n}$, where $\mathcal{S} \subset \mathbb{N}^n$ is the support of $f$ and each $I=(i_1,\ldots,i_n) \in \mathcal{S}$ is a multi-index.
Since $t_1+t_n^{d^{n-1}},t_2+t_n^{d^{n-2}},\ldots,t_{n-1}+t_n^{d}$ are elements in $F[t_1,\ldots,t_n]$, we can apply the substitution
\begin{align*}
&f\left(t_1+t_n^{d^{n-1}},t_2+t_n^{d^{n-2}},\ldots,t_{n-1}+t_n^{d},t_n\right)\\
&	\hspace{2cm}=\sum\limits_{I \in \mathcal{S}} 
		\left( a_{I} t_n^{i_1d^{n-1}+i_2d^{n-2}+\ldots+i_{n-1}d+i_n} + \text{ terms of lower degrees in } t_n \right).
\end{align*}
We see that the coordinates of $I$ are non-negative integers between $0$ and $d-1$, and the exponent $i_1d^{n-1}+i_2d^{n-2}+\ldots+i_{n-1}d+i_n$ in the right hand side of the above expression is exactly the $d$-adic representation. Therefore, among $I$ in $\mathcal{S}$, there is only one multi-index, say $I_0$, that reaches the maximal exponent. Thus $a_{I_0}^{-1}f\left(t_1+t_n^{d^{n-1}},t_2+t_n^{d^{n-2}},\ldots,t_{n-1}+t_n^{d},t_n\right)$ is a monic polynomial in $t_n$.
\end{proof}

We can now prove Theorem \ref{main1} of the introduction:

\begin{theorem}\label{prop:normalization_central}
Let $D$ be a division algebra. The tuple $(\mathbf{id}_D,\ldots,\mathbf{id}_D)$ of identity maps is centrally normalizable over $D$.
Equivalently, every centrally finitely generated extension over $D$ is centrally normalizable over $D$.
\end{theorem} 

\begin{proof}
The proof is the same as the proof of Theorem~\ref{thm:normalization_one_auto}, except using the transformation $t_i \mapsto t_i+t_n^{r^i}$ and Lemma~\ref{lem:F-finite} instead of the transformation $t_i \mapsto t_i+b_it_n$ for $i=1,\ldots,n-1$ and Lemma~\ref{lem:poly_f} (here $r=1+\deg f$).
\end{proof}

In the case where $D$ is a centrally finite division algebra (i.e. $D$ is finite over its center $F=\mathcal{Z}(D)$), and where $S$ is a subring of a ring of central polynomials over $D$, one can further drop the ``central" condition in the hypothesis of Theorem~\ref{prop:normalization_central} (see Proposition~\ref{main} below).

\begin{lemma}\label{center}
Let $D$ be a centrally finite division algebra with center $F$, and let $R=D[t_1,\ldots,t_n]$. For any polynomial $p \in R$, there exist $p_1,\ldots,p_m \in F[t_1,\ldots,t_n]$ such that $D[p] = D[p_1,\ldots,p_m]$ inside $R$.
\end{lemma}

\begin{proof} 
Write $p = \sum_{I} p_I t^I$, where $t = (t_1,\ldots,t_n)$ and $I$ is a multi-index. 
Let $v_1,\ldots,v_m$ be a basis for $D$ over $F$ and write $p_I = \sum_{i=1}^m p_{I,i}v_i$ with $\{p_{I,i}\} \subseteq F$. 
Let $p_i = \sum p_{I,i} x^I \in F[t_1,\ldots,t_n]$, $1 \leq i \leq m$. 
Since the variables $t_1,\ldots,t_n$ are central in $R$, we have 
$$p = \sum_I \big(\sum_i p_{I,i} v_i\big ) t^I 
	=  \sum_{i} \big(\sum_I p_{I,i}t^I\big) v_i 
		= \sum_i p_i v_i.$$
Hence $p \in D[p_1,\ldots,p_m]$.

By \cite[Corollary~4]{Alon2021quaternionic} (a result previously discovered by Wilczy{\' n}ski \cite[Theorem 4.1]{Wil2014}), there exist elements $b_{st}^i \in F$, $1 \leq i,s,t,\leq m$, such that 
$$p_{I,i} = \sum_{s,t = 1}^mb_{st}^i v_s p_I v_t.$$
We thus have 
$$p_i = \sum_I p_{I,i} t^I 
	= \sum_I\sum_{s,t}b_{st}^i v_s p_I t^I v_t 
	= \sum_{s,t}b_{s,t}^i v_s(\sum_I p_I t^I) v_t 
	= \sum_{s,t} b_{s,t}^i v_s p v_t.$$
Hence $p_i \in D[p]$ for $1 \leq i \leq m$. 
\end{proof}



\begin{corollary}\label{cor_center} 
Let $D$ be a centrally finite division algebra, let $R=D[t_1,\ldots,t_n]$, let $A$ be a finite set of polynomials in $R$, and let $S = D[A]$. Then $S$ is centrally finitely generated over $D$. 
\end{corollary} 
\begin{proof} 
Write $A = \{p_1,\ldots,p_k\}$. By Proposition~\ref{center}, there exist polynomials $p_{ij} \in F[t_1,\ldots,t_n]$, such that $D[p_i] = D[p_{i1},\ldots,p_{im_i}]$. It follows that $S = D[p_{11},\ldots,p_{1m_1},\ldots,p_{k1},\ldots,p_{km_k}]$.
\end{proof} 

From Corollary \ref{cor_center} we can deduce Noether's normalization for finitely generated subrings of rings of polynomials in central variables over a centrally finite division algebra $D$ (Theorem \ref{main_centrally_finite} of the introduction). The key point in the proof is that such rings must be {\bf centrally} finitely generated over $D$ (which follows from Lemma~\ref{center}).

\begin{proposition}\label{main} 
Let $D$ be a centrally finite division algebra and let $R=D[t_1,\ldots,t_n]$.
Then for every finite subset $A$ of $R$, the ring $S=D[A]$ is centrally normalizable over $D$. 
\end{proposition}

\begin{proof} 
By Corollary~\ref{cor_center}, $S$ is centrally finitely generated over $D$, hence the claim follows from Theorem~\ref{prop:normalization_central}. 
\end{proof}

A simple application of Theorem~\ref{prop:normalization_central} is the following Nullstellensatz for two-sided ideals in $\mathbb{H}[t_1,\ldots,t_n]$, where $\mathbb{H}$ is the real quaternion algebra. 
The proof is similar to the proof via Zariski's lemma of the classical Nullstellensatz for algebraically closed fields.

\begin{proposition} 
Let $R = \mathbb{H}[t_1,\ldots,t_n]$, the ring of polynomials in $n$ central variables over $\mathbb{H}$. Let $M$ be a two-sided ideal in $R$, which is maximal as a left ideal. Then there exists $a_1,\ldots,a_n \in \mathbb{R}$ such that $t_1-a_1,\ldots,t_n-a_n$ generate $M$. 
\end{proposition} 

\begin{proof} 
The assumption that $M$ is maximal as a left ideal means that $R/M$ is a division algebra, which contains $\mathbb{H}$ as a subring. 
Moreover, $R/M$ is centrally finitely generated over $\mathbb{H}$ by $t_1+M,\ldots,t_n+M$. 
It follows from Proposition~\ref{prop:normalization_central} that $R/M$ is centrally normalizable over $\mathbb{H}$.
Thus, by Proposition \ref{Zari}, $R/M$ is a finite extension of $\mathbb{H}$. 
This implies that $R/M = \mathbb{H}$, by a theorem of Frobenius \cite{frobenius1878uber}, and hence there exist $a_1,\ldots,a_n \in \mathbb{H}$ such that $t_1-a_1,\ldots,t_n-a_n \in M$. By \cite[Lemma 2.1]{Alon2021central} and \cite[Proposition 2.2]{Alon2021central} (the proof of both of these claims is short and elementary), this means that $a_ia_j = a_ja_i$ for all $1 \leq i,j \leq n$ and that $M$ is generated by $t_1-a_1,\ldots,t_n-a_n$. Finally, since $M$ is two-sided ideal, we must have $a_1,\ldots,a_n \in \mathbb{R}$. Indeed, if say $a_1 \notin \mathbb{R}$, then there exist $0 \neq b \in \mathbb{H}$ such that $a_1b \neq ba_1$. Then $(t_1-a_1)b = bt_1-a_1b = b(t_1-b^{-1}a_1b) \in M$, hence $(t_1-b^{-1}a_1b) \in M$. But this means that $t_1-b^{-1}a_1b$ must vanish at $a_1$, that is, $a_1 = b^{-1}a_1b$, hence $ba_1 = a_1b$. \end{proof}

We note that \cite[Proposition 2.6]{Alon2021central} goes further and characterizes all maximal left ideals in $\mathbb{H}[t_1,\ldots,t_n]$, not only the two-sided ones.

\section{Normalizability of general tuples}\label{sec:general_tuples}

In the previous sections, we have shown that constant tuples of automorphisms that are automorphically normalizable over a division algebra. In this section, we consider arbitrary tuples (of finite length) and give some negative examples. In the case where our base ring $D$ is a field, we give a necessary and sufficient condition for such a tuple to be automorphically normalizable.

\subsection{Basic properties}

\begin{proposition}\label{prop:sigma_reduce_exponents}
If a tuple $(\sigma_1,\ldots,\sigma_n)$ of commuting automorphisms in $\mathbf{Aut}(D)$ is not automorphically normalizable over $D$, then the tuple $(\sigma_1^{d_1},\ldots,\sigma_n^{d_n})$ is also not automorphically normalizable over $D$, for any sequence of positive integers $d_1,\ldots,d_n$.
\end{proposition}

\begin{proof}
Assume that the tuple $(\sigma_1,\ldots,\sigma_n)$ is not automorphically normalizable over $D$.
Then there exists an automorphic extension $S=D[z_1,\ldots,z_n;\sigma_1,\ldots,\sigma_n]$ of $D$ that is not automorphically normalizable over $D$.
Observe that $R=D[z_1^{d_1},\ldots,z_n^{d_n};\sigma_1^{d_1},\ldots,\sigma_n^{d_n}]$ is an automorphic extension of $D$ and $S$ is finite over $R$ (with the left linear basis being the set of monomials in $z_1,\ldots,z_n$ of degrees at most $d_1,\ldots,d_n$, respectively).
Therefore, according to Lemma~\ref{lem:finite-over-normalizable}, $R$ is not automorphically normalizable over $D$. As a consequence, the tuple $(\sigma_1^{d_1},\ldots,\sigma_n^{d_n})$ is not automorphically normalizable over $D$.
\end{proof}

\begin{proposition}\label{prop:sigma-prolong}
Let $\sigma_1,\ldots,\sigma_n$ be commuting automorphisms in $\mathbf{Aut}(D)$.
If the tuple $(\sigma_1,\ldots,\sigma_r)$ is not automorphically normalizable over $D$ for some integer $r \leq n$, then so is the tuple $\Sigma=(\sigma_1,\ldots,\sigma_n)$.
\end{proposition}

\begin{proof}
Let $S=D[z_1,\ldots,z_r;\sigma_1,\ldots,\sigma_r]$ be an automorphic extension of $D$ which is not automorphically normalizable over $D$.
Set $z_{r+1}=\ldots=z_n=0$. Then the ring $D[z_1,\ldots,z_n;\sigma_1,\ldots,\sigma_n]=S$ is of course an automorphic extension of $D$ and it is not automorphically normalizable over $D$.
Therefore, $\Sigma$ is not automorphically normalizable over $D$.
\end{proof}

By Propositions~\ref{prop:sigma_reduce_exponents} and~\ref{prop:sigma-prolong}, if we have identified a tuple of commuting automorphisms of $D$ that is not automorphically normalizable over $D$, we can generate infinitely many additional tuples that are not automorphically normalizable over $D$ by prolonging the tuple, or by introducing positive powers to each automorphism in the tuple. However, we have not seen any negative  examples to normalizability thus far. In the following, we provide two such examples, in Propositions~\ref{prop:counter-example-1} and~\ref{prop:counter-example2}.

\subsection{First negative example - the tuple $(\sigma,\sigma^{-1})$}

Let $D$ be a division algebra. Recall that for a nonzero element $c \in D$, the notation $\mathbf{in}_c$ stands for the inner automorphism of $D$ with respect to $c$ (Example~\ref{ex:inner_automorphic}).

\begin{definition}[{see \cite[p.~43]{Lam1994}}]\label{def:quasi-finite-auto}
An automorphism $\sigma \in \mathbf{Aut}(D)$ is called of \textbf{finite inner order} if $\sigma^k$ is an inner automorphism for some positive integer $k$ (i.e. $\sigma^k = \mathbf{in}_c$ for some nonzero element $c \in D$).
Otherwise, we say that $\sigma$ is of \textbf{infinite inner order}.
\end{definition}

If $\sigma \in \mathbf{Aut}(D)$ has finite order $n$, then $\sigma^n = \textbf{id}_D = \textbf{in}_{1}$. Hence, every finite order automorphism has finite inner order but the converse does not hold in general.
In case, $\sigma$ is of infinite inner order, then $\mathbf{in}_{a} \circ \sigma^i \neq \mathbf{in}_{b} \circ  \sigma^j$ for all nonzero elements $a$ and $b$ of $D$.

\begin{example}
Let $D=\mathbb{H}$ be the real quaternion algebra.
Let $\alpha$ be an irrational number and set $u=e^{\mathbf{i} \alpha}$, where $\mathbf{i}$ is the imaginary unit of $\mathbb{C}$. 
Then for any integer $k$, the number $u^k = e^{\mathbf{i} k \alpha}$ is not a real number.
Consider the inner automorphism $\sigma = \mathbf{in}_u$ of $\mathbb{H}$, i.e. $\sigma(a) = uau^{-1}$ for any $a \in \mathbb{H}$.
Of course, $\sigma$ is of finite inner order (taking $n=1$ in the above definition). 
However, $\sigma$ is not of finite order. 
Indeed, if there was an integer $n$ such that $\sigma^n = \mathbf{in}_{u^n}$ is the identity, then $u^n$ would belong to the center $\mathcal{Z}(\mathbb{H})=\mathbb{R}$, a contradiction.
\end{example} 

\begin{lemma}\label{lem:linearly_separable}
Let $S/D$ be a ring extension.
Assume that elements $f_1,\ldots,f_d \in S$ are automorphic and left linearly independent over $D$.
Then the sum $f=f_1+\ldots+f_d$ is automorphic over $D$ if and only if $f,f_1,\ldots,f_d$ are automorphic over $D$ with respect to the same automorphism.
\end{lemma}

\begin{proof}
Assume that $f_i$ is automorphic over $D$ with respect to $\pi_i \in \mathbf{Aut}(D)$ for each $i=1,\ldots,d$. 
The ``if" part is trivial.

For the ``only if" part, let us suppose that $f$ is automorphic over $D$ with respect to $\pi \in \mathbf{Aut}(D)$.
Then, for every $b \in D$, we have
\begin{align*}
0=fb-\pi(b)f &= (f_1+\ldots+f_d) b-\pi(b)(f_1+\ldots+f_d) \\
	&= (\pi_1(b)-\pi(b)) f_1 + \ldots + (\pi_n(b)-\pi(b)) f_n.
\end{align*}
Since $f_i$s are left linearly independent over $D$, we must have $\pi_1(b)=\pi(b)$ for all $i$.
Thus, $\pi_i=\pi$ for all $i$.
In other words, $f,f_1,\ldots,f_d$ are automorphic over $D$ with respect to the same automorphism $\pi$.
\end{proof}

We will now give the first negative example for the normalizability of an automorphic extension of a division algebra.
This example is chosen to be a subring of a division algebra $D(t;\sigma)$, which is the (left) fraction division algebra of the skew polynomial ring $D[t;\sigma]$.
We note that that this (left) fraction division algebra exists since skew polynomial rings are Ore domains (see \cite[Corollary~6.7]{Goodearl2004}).

\begin{lemma}\label{lem:counter-example_1}
Let $\sigma \in \mathbf{Aut}(D)$ be an automorphism.
Let $S=D[t,t^{-1};\sigma,\sigma^{-1}]$ be the subring generated by $t$ and $t^{-1}$ of the division algebra $D(t;\sigma)$.
Then $S$ is automorphically normalizable over $D$ if and only if $\sigma$ is of finite inner order.
\end{lemma}

\begin{proof}
Let us assume that $\sigma$ is of finite inner order, i.e. there exist a positive integer $k$ such that $\sigma^k=\mathbf{in}_c$, where $\mathbf{in}_c$ is the inner automorphism of $D$ with respect to a nonzero element $c \in D$. 
We will prove that $S$ is automorphically normalizable over $D$.
Indeed, let $u=t^{-k}+c^{-2}t^k$ be an element in $S$.
It is clear that $u$ is algebraically independent over $D$.
In addition, for every $b \in D$, we have
\begin{align*}
	ub&=t^{-k}b+c^{-2}t^kb=\sigma^{-k}(b)t^{-k}+c^{-2}\sigma^k(b)t^k\\
	&=c^{-1}bct^{-k}+c^{-1}bc^{-1}t^k 
	=c^{-1}bc(t^{-k}+c^{-2}t^k)
	=\mathbf{in}_{c^{-1}}(b)u.
\end{align*}
This means that $u$ is automorphic over $D$ with respect to the inner automorphism $\mathbf{in}_{c^{-1}}$.
Furthermore, we see that
$c^{-2}t^{2k}-ut^{k}+1=0$.
Thus, $t$ is left integral over $D[u]$, and $D[t,t^{-k}]=D[t,u]$ is finite over $D[u]$.
Since $S=D[t,t^{-1}]$ is finite over $D[t,t^{-k}]$, $S$ is also finite over $D[u]$.
Hence, $S$ is automorphically normalizable over $D$.

Next, let us assume that $\sigma$ is of infinite inner order.
We will prove that $S$ cannot be automorphically normalizable over $D$.
Assume that there exist a non-negative integer $d$ and commuting elements $a_1,\ldots,a_d$ in $S$ that are algebraically independent over $D$, automorphic over $D$ with respect to a commuting automorphisms $\sigma_1,\ldots,\sigma_d \in \mathbf{Aut}(D)$, and such that $S$ is finite over the subring $R=D[a_1,\ldots,a_d;\sigma_1,\ldots,\sigma_d]$, and $R$ is a skew polynomial ring over $D$.

First, we claim that an element $a$ in $S$ is automorphic over $D$ only if it is of the form $a=bt^k$ for some $b \in D$ and an integer $k$.
Indeed, as an element of $S$, $a$ must be of the form
$a = \sum_{k \in \mathbb{Z}} b_k t^k$
for some $b_k \in D$, where the sum is taken with finitely many non-zero coefficients.
Each nonzero summand $b_kt^k$ in this presentation of  $a$ is automorphic over $D$ with respect to the automorphism $\mathbf{in}_{b_k} \circ \sigma^k$. 
Now let us assume that $a$ is automorphic over $D$.
According to Lemma~\ref{lem:linearly_separable}, all nonzero summands appearing in $a$ must be automorphic over $D$ with respect to the same automorphism.
However, since $\sigma$ is of infinite inner order, we have $\mathbf{in}_{b_k} \circ \sigma^i \neq \mathbf{in}_{b_j} \circ  \sigma^j$ for all integers $i \neq j$ such that $b_i$ and $b_j$ are both nonzero.
Thus there can be only one non-zero summand in the given presentation of  $a$, that is $a=b_kt^k$ for some integer $k$, as claimed.

Thus, since $a_1,\ldots,a_d$ are automorphic over $D$, each $a_i$ is of the form $b_i t^{k_i}$ with $b_i \in D$. Then, on the one hand, $d$ must be nonzero, since $S$ is infinite over $D$. On the other hand, $d$ cannot be greater than $1$, since any two such nonzero terms $b_i t^{k_i},b_j t^{k_j}$ are algebraically dependent over $D$ (recall that $R=D[a_1,\ldots,a_d;\sigma_1,\ldots,\sigma_d]$ is a skew polynomial ring over $D$) . Thus $d=1$.

Suppose that $S$ is finite over $R=D[a_1]$. Write $a_1=bt^k \in S$ for some nonzero element $b \in D$ and a nonzero integer $k$. Without loss of generality, we assume further that $b=1$ and $k>0$.
Thus $R=D[t^k]$. Observe that $R$ is left Noetherian and $S$ is finite over $R$.
It follows from \cite[Lemma~2.2]{ParanVo2023} that $S$ is left integral over $R$.
In particular, $t^{-1}$ is left integral over $R$, i.e. there exist a positive integer $m>0$ and elements $f_0,\ldots,f_{m-1} \in R$ such that
\begin{align*}
	t^{-m}+f_{m-1}t^{-m+1}+\ldots+f_1t^{-1}+f_0=0.
\end{align*}
By multiplying both side by $t^{m-1}$, we obtain
\begin{align*}
	t^{-1} = - \left( f_{m-1}+\ldots+f_1t^{m-2}+f_0t^{m-1} \right)
\end{align*}
which is an element of $D[t]$, a contradiction.
Hence, $S$ cannot be finite over $R$, and so $S$ is not automorphically normalizable over $D$.
\end{proof}

\begin{proposition}\label{prop:counter-example-1}
If $\sigma \in \mathbf{Aut}(D)$ is of infinite inner order, then the tuple $(\sigma,\sigma^{-1})$ is not automorphically normalizable over $D$.
\end{proposition}

\begin{proof}
Straightforward from Lemma~\ref{lem:counter-example_1}.
%
\end{proof}

\begin{question} \label{q1} What can be said concerning normalizability of the tuple $(\sigma,\sigma^{-1})$ when $\sigma$ is of finite inner order?
Corollary~\ref{cor:normalization_cyclic_positive} gives a partial answer when $\sigma$ is of finite-order.
In the case where $\sigma$ is of finite inner order but not of finite order, the authors do not know.
\end{question}

\subsection{Second negative example}

To give our second negative example, we first need the following lemma.

\begin{lemma}\label{lem:counter-example2-pre}
Let $\sigma_1$ and $\sigma_2$ be two commuting automorphisms in $\mathbf{Aut}(D)$.
In the skew polynomial ring $T=D[t_1,t_2;\sigma_1,\sigma_2]$, let $I = (t_1t_2)$ be the two-sided ideal generated by $t_1t_2$.
Let $S=T/I$ be the quotient ring. Then every two commuting elements of $S$ are algebraically dependent over $D$.
\end{lemma}

\begin{proof}
Let $z_1$ and $z_2$ be the images of $t_1$ and $t_2$ in $S$, respectively.
Then $z_1z_2=0$ and $S=D[z_1,z_2;\sigma_1,\sigma_2]$.
Assume that $x_1$ and $x_2$ are two arbitrary elements of $S$ which are the images of two skew polynomials $f_1$ and $f_2$ in $T$, respectively (i.e. $x_i=f_i+I$, $i=1,2$).
Let $d$ be the maximal total degrees of $f_1$ and $f_2$, and let $N$ be a positive integer.
We denote by $T_{dN}$ the set of polynomials in $T$ of total degree at most $dN$, and let $S_{dN}$ be the image of $T_{dN}$ in $S$.
We consider the following two sets:
$$\mathcal{A} = \{x_1^{k_1}x_2^{k_2} \,|\, 0 \leq k_1+k_2 \leq N\},$$
which is the set of all monomials in $x_1$ and $x_2$ of total degree at most $N$, thus is a subset of $S_{dN}$,
and 
$$\mathcal{B} = \{1,z_1,z_2,z_1^2,z_2^2,\ldots,z_1^{dN},z_2^{dN}\},$$
which is a left linear basis of $S_{dN}$ as a vector space over $D$.
If there are two pairs of non-negative integers $(i_1,i_2)$ and $(j_1,j_2)$ with $i_1+i_2,\,j_1+j_2 \leq dN$ such that $x_1^{i_1}x_2^{i_2}=x_1^{j_1}x_2^{j_2}$, then the two monomials $x_1^{i_1}x_2^{i_2}$ and $x_1^{j_1}x_2^{j_2}$ are linearly dependent over $D$.
Thus $x_1$ and $x_2$ is algebraically dependent over $D$.

Let us consider the case when $x_1^{i_1}x_2^{i_2} \neq x_1^{j_1}x_2^{j_2}$ for all pairs of nonnegative integers $(i_1,i_2) \neq (j_1,j_2)$ with $i_1+i_2,\,j_1+j_2 \leq dN$.
In this case, we have
\begin{align*}
\frac{|\mathcal{A}|}{|\mathcal{B}|} 
	&= \frac{\binom{N+2}{2}}{2dN+1} 
	= \frac{(N+2)(N+1)}{2(2dN+1)}
	\geq \frac{(N+2)(N+1)}{2(2d+1)N}
	\geq \frac{N+2}{2(2d+1)}.
\end{align*} 
Therefore, by selecting $N>4d$, we will have $|\mathcal{A}| > |\mathcal{B}|$, thus $\mathcal{A}$ is left linearly dependent over $D$.
Hence, $x_1$ and $x_2$ are algebraically dependent over $D$.
\end{proof}

We now have:

\begin{lemma}\label{lem:counter-example2}
Let $\sigma_1$ and $\sigma_2$ be two commuting automorphisms in $\mathbf{Aut}(D)$.
In the skew polynomial ring $T=D[t_1,t_2;\sigma_1,\sigma_2]$, let $I = (t_1t_2)$ be the two-sided ideal generated by $t_1t_2$.
Then the quotient ring $S=T/I$ is automorphically normalizable over $D$ if and only if $\sigma_1^{k_1} \circ \sigma_2^{-k_2}$ is an inner automorphism of $D$ for some positive integers $k_1,k_2$.
\end{lemma}

\begin{proof}
Let $z_1$ and $z_2$ be the images of $t_1$ and $t_2$ in $S$, respectively.
Then $z_1z_2=0$ and $S=D[z_1,z_2;\sigma_1,\sigma_2]$.
We consider the following two cases.

\noindent
\textbf{Case 1}: $\sigma_1^{k_1} \circ \sigma_2^{-k_2}$ is an inner automorphism of $D$ for some positive integers $k_1,k_2$.
	This means that $\sigma_1^{k_1} \circ \sigma_2^{-k_2}=\mathbf{in}_c$ for some nonzero element $c \in D$.
	We need to prove that $S$ is automorphically normalizable over $D$.
	Indeed, set $u=z_1^{k_1}+cz_2^{k_2}$.
	Then, it is clear that $u$ is algebraically indepedent over $D$.
	In addition, for every $r \in D$, we have
	\begin{align*}
		ur=z_1^{k_1}r+cz_2^{k_2}r
		&=\sigma_1^{k_1}(r) z_1^{k_1} + c \sigma_2^{k_2}(r)z_2^{k_2}
		=\sigma_1^{k_1}(r) z_1^{k_1} + \mathbf{in}_c \circ \sigma_2^{k_2}(r) \cdot c z_2^{k_2}\\
		&=\sigma_1^{k_1}(r)(z_1^{k_1}+cz_2^{k_2})
		=\sigma_1^{k_1}(r)u.
	\end{align*}
	Therefore, $u$ is automorphic over $D$ with respect to $\sigma_1^{k_1}$.
	Furthermore, since $z_1^{k_1+1}=uz_1$, $z_1$ is left integral over $D[u]$.
	As a consequence, $D[z_1,z_2^{k_2}] = D[z_1,cz_2^{k_2}] = D[z_1,u]$ is finite over $D[u]$.
	Since $S=D[z_1,z_2]$ is finite over $D[z_1,z_2^{k_2}]$, $S$ is also finite over $D[u]$.
	Hence, $S$ is automorphically normalizable over $D$.

    \noindent
    \textbf{Case 2}: $\sigma_1^{k_1} \circ \sigma_2^{-k_2}$ is not an inner automorphism of $D$ for arbitrary positive integers $k_1,k_2$.
	We will prove that $S$ is not automorphically normalizable over $D$.
	Assume not, then according to Lemma~\ref{lem:counter-example2-pre}, there exist an automorphic element $x \in S$ such that $x$ is algebraically independent over $D$ and that $S$ is finite over $D[x]$.
	
	As an element of $S$, $x$ can be written in the form
	\begin{align}\label{eq:counter-example2}
		x=a_0 + \sum\limits_{i \geq 1}a_iz_1^i + \sum\limits_{j \geq 1} b_jz_2^j,
	\end{align}
	where the sums are finite and the nonzero terms are left linearly independent over $D$.
	The terms $a_iz_1^i$ and $b_jz_2^j$ are automorphic over $D$ with respect to the automorphisms $\mathbf{in}_{a_i} \circ \sigma_1^i$ and $\mathbf{in}_{b_j} \circ \sigma_2^{j}$, respectively.
	Observe that $\mathbf{in}_{a_i} \circ \sigma_1^i \neq \mathbf{in}_{b_j} \circ \sigma_2^{j}$ for all positive integers $i,j$ and nonzero elements $a_i,b_j$ of $D$.
	Otherwise, $\sigma_1^i \circ \sigma_2^{-j} = \mathbf{in}_{a_i^{-1}b_j}$ is an inner automorphism of $D$, which is a contradiction.
	This means that, in Eq.~\eqref{eq:counter-example2}, no nonzero term in the first sum is automorphic over $D$ with respect to the same automorphism with any nonzero term of the second sum.
	Therefore, as a consequence of Lemma~\ref{lem:linearly_separable}, in order $x$ to be automorphic over $D$, one of the two sums in Eq.~\eqref{eq:counter-example2} is zero.
	Thus, either $x \in D[z_1]$ or $x \in D[z_2]$.
	
	Without loss of generality, let us assume that $x \in D[z_1]$.
	Then $D[x] \subseteq D[z_1]$.
	Since $S$ is finite over $D[x]$, $S$ is also finite over $D[z_1]$.
	According to \cite[Lemma~2.2]{ParanVo2023}, $S$ is left integral over $D[z_1]$.
	In particular, $z_2$ is left integral over $D[z_1]$, i.e. there exist a positive integer $d$ and elements $f_0,\ldots,f_{d-1} \in D[z_1]$ such that
	\begin{align*}
		z_2^{d}+f_{d-1}z_2^{d-1} + \ldots + f_1z_2+f_0=0.
	\end{align*}
	Since $z_1z_2=0$, by replacing $f_1,\ldots,f_{d-1}$ by their free coefficients, we can assume that $f_1,\ldots,f_{d-1} \in D$.
	Then 
	$$f_0=-(z_2^{d}+f_{d-1}z_2^{d-1} + \ldots + f_1z_2)$$
	is an element in $D[z_1] \cap D[z_2]=D$.
	Thus, $z_2$ is left integral over $D$, which is impossible.
	Hence, $S$ is not automorphically normalizable over $D$.
\end{proof}

\begin{proposition}\label{prop:counter-example2}
Let $\sigma_1$ and $\sigma_2$ be two commuting automorphisms in $\mathbf{Aut}(D)$.
If $\sigma_1^{k_1} \circ \sigma_2^{-k_2}$ is not an inner automorphism for every positive integers $k_1,k_2$, then the tuple $(\sigma_1,\sigma_2)$ is not automorphically normalizable over $D$.
\end{proposition}

\begin{proof}
Straightforward from Lemma~\ref{lem:counter-example2}.
\end{proof}

\begin{question}\label{q2}
What can be said concerning the normalizability of the tuple $(\sigma_1,\sigma_2)$ if $\sigma_1^{k_1} \circ \sigma_2^{-k_2}$ is an inner automorphism for some positive integers $k_1,k_2$?
Corollary~\ref{cor:normalization_cyclic_positive} gives a partial answer for the case where $\sigma_1^{k_1} \circ \sigma_2^{-k_2}=\mathbf{id}_D$.
\end{question}

In the case where $D=F$ is a field, we give a necessary and sufficient condition for a tuple of commuting automorphisms in $\mathbf{Aut}(F)$ to be automorphically normalizable in the following theorem (Theorem \ref{fields_iff} of the introduction).

\begin{theorem}\label{thm:notrivial_case}
Let $F$ be a field and $\sigma_1,\ldots,\sigma_n$ be commuting automorphisms in $\mathbf{Aut}(F)$ such that $\cap_{i=1}^n F_{\sigma_i}$ is infinite.
Then the tuple $(\sigma_1,\ldots,\sigma_n)$ is automorphically normalizable over $F$ if and only if $\sigma_1^{d_1}=\ldots=\sigma_n^{d_n}$ for some positive integers $d_1,\ldots,d_n$.
\end{theorem}

\begin{proof}
The ``if" part is a direct consequence of Corollary~\ref{cor:normalization_cyclic_positive}.
For the ``only if" part, let us assume that the tuple $(\sigma_1,\ldots,\sigma_n)$ is automorphically normalizable over $F$.
According to Proposition~\ref{prop:sigma-prolong} and Remark~\ref{rem:normalizability}, the shorter tuple $(\sigma_1,\sigma_i)$ is also automorphically normalizable over $F$ for every $i=2,\ldots,n$.
Then it follows from Proposition~\ref{prop:counter-example2} that there exist positive integers $k_{i,1}$ and $k_{i,2}$ such that $\sigma_1^{k_{i,1}} \circ \sigma_i^{-k_{i,2}}$ is an inner automorphism.
However, every inner automorphism on $F$ is the identity map.
Therefore, $\sigma_1^{k_{i,1}} \circ \sigma_i^{-k_{i,2}} = \mathbf{id}_F$, or equivalently, $\sigma_1^{k_{i,1}} = \sigma_i^{k_{i,2}}$, for all $i=2,\ldots,n$.
Set $d_1=\prod_{i=1}^{n}k_{i,1}$, then we have $\sigma_1^{d_1} = \sigma_i^{\frac{d_1k_{i,2}}{k_{i,1}}}$.
Thus, for each $i=2,\ldots,n$, if we set $d_i=\frac{d_1k_{i,2}}{k_{i,1}}$, then $d_1,\ldots,d_n$ are positive integers and
$\sigma_1^{d_1}=\ldots=\sigma_{n}^{d_n}$
as expected.
\end{proof}

\begin{question}\label{q3}
Can one find a necessary and sufficient condition for the above theorem to hold, when the field $F$ is replaced by a division algebra? 
\end{question}

\section*{Acknowledgement}

The authors wish to express their sincere thanks to the anonymous referee for valuable suggestions that 
improved the final version of the manuscript.

\end{document}